\documentclass{amsart}
\usepackage[T1]{fontenc}
\usepackage{tikz-cd}
\usepackage{bbm}
\usepackage{cjhebrew}
\usepackage{amsmath, amssymb}
\usepackage{amsthm} 
\usepackage{amsfonts}
\usepackage{amscd}\usepackage{amsxtra}
\usepackage{multirow}
\usepackage[all,cmtip]{xy}
\usepackage{amsrefs}
\usepackage{hyperref}
\usepackage[cmtip,all]{xy}
\usepackage{listings}
\usepackage{color}

\definecolor{dkgreen}{rgb}{0,0.6,0}
\definecolor{gray}{rgb}{0.5,0.5,0.5}
\definecolor{mauve}{rgb}{0.58,0,0.82}

\usepackage{yhmath} 

\newcommand{\red}{{\text{\rm red}}}

\newcommand{\sm}{{\text{\rm sm}}}
\newcommand{\sing}{{\text{\rm sing}}}

\newcommand{\sA}{\mathcal{A}}

\newcommand{\sG}{\mathcal{G}}

\newcommand{\sL}{\mathcal{L}}

\newcommand{\sN}{\mathcal{N}}
\newcommand{\sO}{\mathcal{O}}

\newcommand{\sY}{\mathcal{Y}}
\newcommand{\sX}{\mathcal{X}}
\newcommand{\bG}{\mathbb{G}}

\newcommand{\bZ}{\mathbb{Z}}

\newcommand{\bL}{\mathbb{L}}

\newcommand{\bQ}{\mathbb{Q}}

\newcommand{\bA}{\mathbb{A}}

\renewcommand \dim[1]{\mbox{dim}{#1}}

\newcommand\sch[1]{\mathbf{Sch}/{#1}}
\newcommand\set{\mathbf{Sets}}

\newcommand\var{\mathbf{Var}}

\newcommand\spec[1]{\operatorname{Spec}(#1)}

\newcommand\fatpoints[1]{\mathbf{Fat}/{#1}}

\newcommand\into{\hookrightarrow}

\newcommand \grot[1]{{\mathbf {Gr}(#1)}}

\catcode`\@=11
\def\opn#1#2{\def#1{\mathop{\kern0pt\fam0#2}\nolimits}} 
\def\underrightarrow{\mathpalette\underrightarrow@}
\def\underrightarrow@#1#2{\vtop{\ialign{$##$\cr
 \hfil#1#2\hfil\cr\noalign{\nointerlineskip}%
 #1{-}\mkern-6mu\cleaders\hbox{$#1\mkern-2mu{-}\mkern-2mu$}\hfill
 \mkern-6mu{\to}\cr}}} 

\def\underleftarrow{\mathpalette\underlefalwaystarrow@}
\def\underleftarrow@#1#2{\vtop{\ialign{$##$\cr
 \hfil#1#2\hfil\cr\noalign{\nointerlineskip}#1{\leftarrow}\mkern-6mu
 \cleaders\hbox{$#1\mkern-2mu{-}\mkern-2mu$}\hfill
 \mkern-6mu{-}\cr}}}
    \let\phi=\varphi
    \let\epsilon=\varepsilon  
\def\:{\colon}   
\let\oldtilde=\tilde
\def\tilde#1{\mathchoice{\widetilde{#1}}{\widetilde{#1}}%
{\indextil{#1}}{\oldtilde{#1}}}
\def\indextil#1{\lower2pt\hbox{$\textstyle{\oldtilde{\raise2pt%
\hbox{$\scriptstyle{#1}$}}}$}}
\def\pnt{{\raise1.1pt\hbox{$\textstyle.$}}}  
\let\amp@rs@nd@\relax
\newdimen\ex@
\newdimen\bigaw@
\newdimen\minaw@
\newdimen\minCDaw@  
\newif\ifCD@
\def\minCDarrowwidth#1{\minCDaw@#1}

\def\@CD{\def\A##1A##2A{\llap{$\vcenter{\hbox
 {$\scriptstyle##1$}}$}\Big\uparrow\rlap%
{$\vcenter{\hbox{$\scriptstyle##2$}}$}&&}%
\def\V##1V##2V{\llap{$\vcenter{\hbox
 {$\scriptstyle##1$}}$}\Big\downarrow\rlap%
{$\vcenter{\hbox{$\scriptstyle##2$}}$}&&}%
\def\={&\hskip.5em\mathrel
 {\vbox{\hrule width\minCDaw@\vskip3\ex@\hrule width
 \minCDaw@}}\hskip.5em&}%
\def\verteq{\Big\Vert&&}%
\def\noarr{&&}%
\def\vspace##1{\noalign{\vskip##1\relax}}%
\relax\let\amp@rs@nd@&\iffalse}\fi
 \CD@true\vcenter\bgroup\relax%
\let\\=\cr\iffalse}\fi\tabskip\z@skip\baselineskip20\ex@
 &\hfill$\m@th##$\hfill\cr}
\def\@endCD{\cr\egroup\egroup}
\def\>#1>#2>{\amp@rs@nd@\setbox\z@\hbox{$\scriptstyle
 \;{#1}\;\;$}\setbox\@ne\hbox{$\scriptstyle\;{#2}\;\;$}\setbox\tw@
 \hbox{$#2$}\ifCD@
 \global\bigaw@\minCDaw@\else\global\bigaw@\minaw@\fi
 \ifdim\wd\z@>\bigaw@\global\bigaw@\wd\z@\fi
 \ifdim\wd\@ne>\bigaw@\global\bigaw@\wd\@ne\fi
 \ifCD@\hskip.5em\fi
 \ifdim\wd\tw@>\z@
 \mathrel{\mathop{\hbox to\bigaw@{\rightarrowfill}}\limits^{#1}_{#2}}\else
 \mathrel{\mathop{\hbox to\bigaw@{\rightarrowfill}}\limits^{#1}}\fi
 \ifCD@\hskip.5em\fi\amp@rs@nd@}
\def\<#1<#2<{\amp@rs@nd@\setbox\z@\hbox{$\scriptstyle
 \;\;{#1}\;$}\setbox\@ne\hbox{$\scriptstyle\;\;{#2}\;$}\setbox\tw@
 \hbox{$#2$}\ifCD@
 \global\bigaw@\minCDaw@\else\global\bigaw@\minaw@\fi
 \ifdim\wd\z@>\bigaw@\global\bigaw@\wd\z@\fi
 \ifdim\wd\@ne>\bigaw@\global\bigaw@\wd\@ne\fi
 \ifCD@\hskip.5em\fi
 \ifdim\wd\tw@>\z@
 \mathrel{\mathop{\hbox to\bigaw@{\leftarrowfill}}\limits^{#1}_{#2}}\else
 \mathrel{\mathop{\hbox to\bigaw@{\leftarrowfill}}\limits^{#1}}\fi
 \ifCD@\hskip.5em\fi\amp@rs@nd@}
\def\@CDS{\def\A##1A##2A{\llap{$\vcenter{\hbox
 {$\scriptstyle##1$}}$}\Big\uparrow\rlap%
{$\vcenter{\hbox{$\scriptstyle##2$}}$}&}%
\def\V##1V##2V{\llap{$\vcenter{\hbox
 {$\scriptstyle##1$}}$}\Big\downarrow\rlap%
{$\vcenter{\hbox{$\scriptstyle##2$}}$}&}%
\def\={&\hskip.5em\mathrel
 {\vbox{\hrule width\minCDaw@\vskip3\ex@\hrule width
 \minCDaw@}}\hskip.5em&}
\def\verteq{\Big\Vert&}
\def\novarr{&}
\def\noharr{&&}
\def\SE##1E##2E{\slantedarrow(0,18)(4,-3){##1}{##2}&}
\def\SW##1W##2W{\slantedarrow(24,18)(-4,-3){##1}{##2}&}
\def\NE##1E##2E{\slantedarrow(0,0)(4,3){##1}{##2}&}
\def\NW##1W##2W{\slantedarrow(24,0)(-4,3){##1}{##2}&}
\def\slantedarrow(##1)(##2)##3##4{\thinlines\unitlength1pt%
\lower 6.5pt\hbox{\begin{picture}(24,18)%
\put(##1){\vector(##2){24}}%
\put(0,8){$\scriptstyle##3$}%
\put(20,8){$\scriptstyle##4$}%
\end{picture}}}
\def\vspace##1{\noalign{\vskip##1\relax}}\relax%
\let\amp@rs@nd@&\iffalse}\fi
 \CD@true\vcenter\bgroup\relax\let\\=\cr\iffalse}\fi\tabskip\z@skip\baselineskip20\ex@
 &\hfill$\m@th##$\hfill\cr}
\def\@endCDS{\cr\egroup\egroup}
\theoremstyle{plain} 
\newtheorem{theorem}{\indent\sc Theorem}[section]
\newtheorem{lemma}[theorem]{\indent\sc Lemma}
\newtheorem{corollary}[theorem]{\indent\sc Corollary}
\newtheorem{proposition}[theorem]{\indent\sc Proposition}

\theoremstyle{definition} 
\newtheorem{definition}[theorem]{\indent\sc Definition}
\newtheorem{remark}[theorem]{\indent\sc Remark}
\newtheorem{example}[theorem]{\indent\sc Example}

\newtheorem{question}[theorem]{\indent\sc Question}


\begin{document}

\title{Auto-arcs of complete intersection varieties.}
\author{Andrew R. Stout}
\address{Borough of Manhattan Community College, CUNY}
\curraddr{199 Chambers Street\\ New York, NY 10007}
\email{astout@bmcc.cuny.edu}
\thanks{Support for this project was provided by a PSC-CUNY Award (PSC-Grant Traditional A, \# 66024-00 54), jointly funded by the Professional Staff Congress and The City University of New York.}


\begin{abstract}
We systematically study the so-called auto-arc spaces. Auto-arc spaces were originally introduced by Schoutens in \cite{sch2} and later generalized and studied by the author in \cite{sto2017}, \cite{auto}, and \cite{sto2019}. In that aforementioned work, only results concerning trivial deformations were explicitly considered because even in that case auto-arc spaces being a subset of generalized jet schemes are difficult to understand. The major advance in this work is obtained by considering auto arc spaces of complete intersections. It is shown that over $k[t]/(t^{n+1})$, these spaces can be viewed as global flat deformations over $\bA_k^n$ of the classical jet scheme of order $n$. We also introduce the study of so-called strong/weak deformations of curves in this context, and we show that a motivic volume can be defined in this case.  
\end{abstract}

\maketitle

\section{Introduction}

The study of jet schemes and arc spaces are an important area of algebraic geometry because these spaces encode a large amount of information above the underlying schemes singular points. 
In this paper, we focus our study on particular types of generalized jet scheme which we term auto-arc spaces. These are defined roughly to be generalized jet schemes of a flat deformation of a scheme over fat point along that same fat point.  Some first motivating examples are discussed in detail in Section \ref{sec1} pertain to the reduced tangent bundle of a scheme over the dual numbers. 

First, we will describe briefly jet schemes and then then the generalization to auto-arc spaces. The {\it jet scheme} of $X$ {of order} $n$ over a field $k$, most commonly denoted in the literature as $\sL_n(X)$, is the scheme with its induced reduced structure which represents the functor, $\sch{k} \to \set$, given by 
$$Y \mapsto \mbox{Hom}_{k}(Y\times_k \spec{k[t]/(t^{n+1}}, Y\times_k X)$$
where $\sch{k}$ denotes schemes over a field $k$.  As a consequence of a theorem due to Grothendieck,  $\sL_n(X)$ exists provided $X$ is separated and locally of finite type over $k$. 
Classically, the {\it arc space of} $X$ {\it over} $k$ in the literature is the projective limit $\sL(X) := \varprojlim \sL_n(X)$ over the natural truncation maps $\pi^n_{n-1} : \sL_n(X) \to \sL_{n-1}(X)$ induced by modding out by $(t^n)\cdot k[t]/(t^{n+1})$. This is also a scheme over $k$ since the transition maps $\pi^n_{n-1}$ are affine. 

The definition of a generalized jet scheme over $k$ is obtained by replacing the linear jet $\spec{k[t]/(t^n)}$ with an arbitrary finitely generated Artinian $k$-algebra. One advantage to this viewpoint is that for an affine $k$-scheme $X=\spec{R}$,  the $k$-points, given by an $a$, on the generalized jet scheme are in one-to-one correspondence with $k$-algebra homomorphisms $\varphi_a : R \to A$ from which an explicit description of the generalized jet scheme as an affine scheme can be obtained.  Similarly then to the classical jet scheme case, the generalized jet scheme of $X$ with respect to $Z=\spec{A}$ exists if $X$ is separated and locally of finite type over $k$. As we discuss below, we will then denote the resulting scheme as $\underline{Hom}_S(Z, X)$ to denote this scheme, and we let $\underline{Hom}_S(Z, X)^{\red}$ denote this scheme with its reduced induce structure.

We also sometimes prefer to work at the following level of generality. Let $S$ be an arbitrary scheme and let $X$ and $X'$ be $S$-schemes locally of finite presentation.  We define the functor from $\sch{S} \to \set$ defined by $$Y \mapsto \mbox{Hom}_S(X'\times_SY, X\times_SY)$$
When this functor is representable by an algebraic space (or more generally by an algebraic stack) over $S$, we write the resulting space as $\underline{Hom}_S(X', X)$. It is a well-known result due to Artin\footnote{Without adding any additional assumptions, we can let $X$ and $X'$ be merely algebraic spaces over $S$.} that the functor is representable by a separated algebraic space over $S$ provided $X'$ is proper and flat over $S$ and $X$ is $S$-separated.

Now, given any connected, finite and flat scheme $Z$ over an arbitrary scheme $S$, we consider an infinitesimal flat deformation $Y \to Z$ of an $S$-scheme $X$ (separated and locally of finite presentation over $S$), and we define the auto-arc space of $Y$ to be the algebraic space over $S$, defined by $\sA_Z(Y) := \underline{Hom}_S(Z,Y)^{\red}$ where $(-)^{\red}$ denotes the reduced structure. When $S$ is the spectrum of a field $k$, then this algebraic space will be a reduced separated scheme locally of finite type over $k$. 
In Section \ref{sec2}, we introduce the reader to some general results concerning these spaces. Notably, we borrow some well-known results concerning jet schemes of complete intersections (initially worked out by Mustata in \cite{mus2001}), and we see that many of these types of results transfer remarkably well.  In both Section \ref{sec2} and Section \ref{sec3}, it becomes apparent that the flat locus of a particular natural morphism, denoted $\theta : \sA_Z(Y) \to \sA_Z$, is of paramount importance. More explicitly, one point of this paper is that the underlying ``relativized notion" of those aforementioned results concerning complete intersections as applied to general auto-arc spaces is best viewed as a question concerning the flat locus of the morphism $\theta$. 

To see this more clearly, in Section \ref{sec3}, we restrict to the case of linear auto-arc spaces, and then in Section \ref{sec4}, we restrict further to linear auto-arc spaces of curves. In either case, we see that these spaces can be regarded as flat global deformations of the classical jet scheme. Many properties of classical jet schemes carry over directly into this context, but crucially, some very well-known results do not or become very subtle (i.e., dependent on choice of deformation). For example, it is no longer automatic that the linear auto arc space of a deformed curve will be reducible if the original curve is singular. We therefore, propose the idea to measure the so-called ``weakness" of a deformation as the ``normalized dimension" of the inverse image of the singular locus in $\sA_Z(Y)$ as a measurement of degeneracy of this behavior. In the remaining sections, we develop the idea of a motivic volume and we show that motivic milnor fiber and its corresponding zeta function have a natural interpretation in this context. 

\section{Tangent bundles of deformations over the dual numbers} \label{sec1}

We let $\mathbf{Sch}$ denote the category of schemes. Given an object $S$ in $\mathbf{Sch}$, we let $\sch{S}$ denote the slice category of schemes $X$ over $S$. Moreover, we let $\fatpoints{S}$ denote the full subcategory of $\sch{S}$ whose objects are formed by all connected schemes $Z$ such that the structure morphism $Z \to S$ is finite and flat. 

Following \cite{liu2006}, we say that a scheme $S$ is a Dedekind scheme if it is a normal locally Noetherian scheme of dimension strictly less than $2$, and we first start with the following observation. 

\begin{theorem}\label{LiuFlat}
Let $X$ be reduced scheme in $\sch{S}$ with $S$ a Dedekind scheme.  The structure morphism $j: X \to S$ is flat if and only if every irreducible component of $X$ dominates $S$.
\end{theorem}

\begin{proof}
Note that an irreducible component $Z$ of $X$ dominates $S$ if the image of $Z$ is dense in $S$ -- i.e., the set-theoretic closure $\overline{j(Z)}$ is equal to $S$. This theorem is a restatement of Proposition 3.9 on page 137 of \cite{liu2006}.
\end{proof}

There are a lot of straightforward conclusions one can reach by using Theorem \ref{LiuFlat}. For now, will just need to consider the following drastically simplified case.  

\begin{corollary} \label{1stcor}
Let $k$ be a field, $X$ be an integral scheme, and  $j: X \to \mathbb{A}_{k}^{1}$ a surjective morphism of schemes. Then, $j$ is flat. 
\end{corollary}


For now, let $D = S \times \spec{\mathbb{Z}[t]/(t^2)}$ be the dual numbers over the scheme $S$. 
Then, $D \to S$ is finite and flat, and as such $\underline{Hom}_{S}(D, X)$ exists in $\sch{S}$ for all objects $X$ in $\sch{D}$ such that the structure morphism $j: X \to D$ is separated and locally of finite presentation\footnote{Thus, the fiber of $X$ at zero defined by $X_0 := X \times_{D_S} S$ is separated and locally of finite presentation.}.  We let $T(X)$ denote {\it the reduced tangent bundle over} $S$, which is defined by $T(X) := \underline{Hom}_{S}(D, X)^{\red}.$

\begin{lemma}\label{surj}
Let $S = \spec{A}$ with $A$ a reduced Noetherian local ring and let $D$ be the dual numbers over $S$. Assume that $X$ is separated and locally of finite type\footnote{Note that a morphism $j : X \to S$ is locally of finite presentation if and only if $S$ is locally noetherian and $j$ is locally of finite type.} over $D$. Then, $T(D) \cong \mathbb{A}_{S}^1$ and moreover, the natural morphism from $T(X)$ to $T(D)$ is surjective provided $X$ admits at least one morphism $S \to X$. 
\end{lemma}

\begin{proof}
The fact that $T(D) \cong \bA_{S}^1$ is proven in Lemma 4.3 on page 141 of \cite{sto2019}. Although it is not difficult, it is slightly less straightforward to show that the natural map $\theta: T(X) \to T(D)$ is surjective. We will provide a proof here as it will be illustrative in regards to other results in this section of the paper.

Therefore, for this second fact, we can assume without loss of generality that $X$ is an affine subscheme of $\bA_{D}^{N}$ defined by equations $F_i = 0$ (for $i = 1, \ldots, s$) of the form $F_i = G_i + tH_i$ where $G_i, H_i$ are elements of $A[x_1, \ldots, x_N]$. Thus, as a subvariety of $\bA_{D}^{2N}$, the equations defining $T(X)$ are given by creating {\it arc variables} $\wideparen{x}_i = y_i + z_it$ and substituting them into $G_i$ to obtain the {\it arc equations} $\wideparen{G}_i +t \wideparen{H}_i=0$. This implies that the equations defining the arc space are of the form
\begin{equation} \label{editeq1}
0 = G_i(y_1, \ldots, y_N)  = t\cdot H_i(y_1, \ldots, y_N) = t\cdot \partial G_i(y_1, \ldots, y_N, z_1, \ldots, z_N)
\end{equation}
where $\partial G_i$ the polynomial such that 
$\wideparen{G}_i = G_i(y_1, \ldots, y_N) + t\cdot\partial G_i$. 

Any point $\alpha$ in $T(D)$ is given by an $a \in A$ defining an endormorphism $t \mapsto a\cdot t$. Thus, any potential point in the fiber $\theta^{-1}(a)$ is given by performing a substitution $at$ for $t$ in Equation \ref{editeq1}. Therefore, any morphism $\beta: D \to  \bA_{D}^{2N}$ which induces $\alpha : D \to D$ can potentially give a well-defined morphism $\beta'$ from $D$ to $T(X)$, and it will be sent to $\alpha$ via the morphism $\theta$. A morphism $S \to X$ will give rise to a morphism $h : S \to T(X)$. Otherwise, $X^{\red}$ would be $S$-smooth since the Jacobian would have no solutions over $A$, and in that case, the statement would be trivial.  Therefore, we let $\beta'$ be any morphism inducing $\alpha$ which factors through $h$. In other words, we have a morphism $\beta'$ which fits into the following commutative diagram:
\[\begin{tikzcd}
 D
 \arrow[drr, bend left , "{\beta'}"]
  \arrow[drrr, bend left, "{\beta}"]
  \arrow[ddrrr, bend right, "{\alpha}"]
  \arrow[dr, "{t =0 }" ] & & \\
&S  \arrow[r,"{h}"]    & T(X) \arrow[r, hook, ""] 
      & \bA_{D}^{2N} \arrow[d, ""] \\
& &\  &D 
\end{tikzcd}\]
 \end{proof}

In particular, in the case that  $S = \spec{k}$ where $k$ is a field, we have the following theorem.

\begin{theorem} \label{tanflat}
Let $S =\spec{k}$ with $k$ a field, let $D$ be the dual numbers over $S$, and assume that $X$ has a  $k$-point. Assume further that $X$ is separated and locally of finite type over $D$ and let $T(X)$ be the reduced tangent bundle. If $T(X)$ is irreducible, then the natural surjective morphism $T(X) \to \mathbb{A}_{k}^{1}$ is flat. 
\end{theorem}

\begin{proof}
This is a direct consequence of Corollary \ref{1stcor} and Lemma \ref{surj}. 
\end{proof}

\begin{example} \label{cubicA}
Consider the infinitesimal deformation $$X = \spec{k[x,y,t]/(y^2 - x^3 - t, t^2)}$$ of the cuspidal cubic $C$ given by $y^2 = x^3$ in $\bA_k^2$. We directly compute the tangent space $T(X)$ by assigning variables $a,b,c,d,e,f$ over $k$ and forming arc variables
\begin{equation*}
\wideparen{x} = a + bt, \quad
\wideparen{y} = c+ dt ,\quad
\wideparen{t} = e + ft
\end{equation*}

Then, the coordinate ring for the tangent space is defined by substituting the arc variables into the defining equations for $X$ and then by setting these equations to zero in $A[t]/(t^2)$, where $A=k[a,b,d,e,f]$. Thus, we have the equations 
\begin{equation*}
\begin{split}
0 &= \wideparen{y}^2 - \wideparen{x}^3 - \wideparen{t} = (c^2-a^3) \cdot 1 + (2cd - 3a^2b - f)\cdot t\\
0 &= \wideparen{t}^2 = e^2\cdot 1 + 2ef \cdot t
\end{split}
\end{equation*}
Since $k[t]/(t^2)$ is a vector space over $k$ with basis $\{1, t\}$, these equations imply that 
\begin{equation*}
\underline{Hom}_k(D, X) \cong \spec{k[a,b,c,d,e,f]/(c^2 - a^3, 2cd - 3a^2b - f, e^2, 2ef)}
\end{equation*}
In the reduced structure $e^2 = 0 \implies e = 0$, which implies 
$T(X)$ is isomorphic to  $\mbox{Spec}(k[a,b,c,d,f]/I)$
where $I = (c^2 - a^3, 2cd - 3a^2b - f)$, which is clearly prime ideal. Therefore, the natural morphism $\theta : T(X) \to \spec{k[t]}$ is flat. 
\end{example}

\begin{remark}
As we will frequently use Singular in this section, we include code which demonstrates how to check that a particular ideal $I$ of $R$ is prime. The interested reader may want to use this code to test alternative examples for which the ideal $I$ is not so clearly prime.

\begin{lstlisting}[language = C]
LIB "primdecint.lib"; //Loads library with primary decomposition procedures

ring R  =  0, (a,b,c,d,f), dp;
ideal I = (c2 - a3, 2cd - 3a2b - f);
ideal J = std(I);
primdecZ(I);
print("Compare the above with the ideal I in standard basis below:");
print(J);
\end{lstlisting}
The output of this code is 
\begin{verbatim}
[1]:
   [1]:
      _[1]=3bc2-2acd+af
      _[2]=3a2b-2cd+f
      _[3]=a3-c2
   [2]:
      _[1]=3bc2-2acd+af
      _[2]=3a2b-2cd+f
      _[3]=a3-c2
Compare the above with the ideal I in standard basis below:
3bc2-2acd+af,
3a2b-2cd+f,
a3-c2
\end{verbatim}
This explicitly shows that $I$ is a prime ideal and thus $T(X)$ is irreducible. \end{remark}

\begin{remark}\label{redcent}
The fact that $T(X)$ is irreducible is in stark contrast to the reduced case -- i.e., the central fiber $\theta_f^{-1}(0) \cong T(C)$ is reducible whenever $C$ is a singular curve for simply dimensional reasons!  We briefly adapt this line of thinking to the case of first order deformations of singular curves in Section \ref{sec4}. Note that in general $T(X)$ itself my be either reducible or irreducible for any first order deformation of a singular curve, yet the central fiber will always be reducible for dimensional reasons. We consider some reducible cases below as in reality the morphism $\theta_f$ will still be flat in this case as well.  \end{remark}

\begin{example} \label{fibex}
Here is a ``non-example" to Theorem \ref{tanflat}. Let $Y$ be a smooth scheme over a field $k$. Then, $Y$ is rigid and hence any deformation $X$ over $D$ is isomorphic to the trivial deformation $Y\times_k D$.
Since $Y$ is smooth, we may cover $Y$ by open affines $U$ such that for some $d$ there is an \'{e}tale morphism $U \to \bA_{k}^d$, which implies that there is an isomorphism $T(U) \cong U \times_k \bA_k^d$ and therefore that $T(Y) \to \spec{k}$ is smooth. Since smoothness is stable under base change $T(Y) \times_k \bA_{k}^1 \to \bA_k^1$ is also smooth.  Moreover, one can directly show that $T(X)$ is \'{e}tale locally isomorphic to $T(Y) \times_k \bA_{k}^1$ and therefore
$T(X) \to \bA_k^1$ is smooth and hence flat.   

If all we want to show is flatness, then we do not need to assume $Y$ is smooth as long as $X$ is the trivial deformation of $Y$ over $X$. This is because $T(Y) \to \spec{k}$ is obviously flat, and since flatness is stable under base change (cf. Proposition 9.2 on page 254 of \cite{har2010}), $T(Y) \times_k \bA_{k}^1 \to \bA_k^1$ is also flat. Moreover, one can generally show that $T(X)$ is \'{e}tale locally isomorphic to $T(Y) \times_k \bA_{k}^1$ from whence it follows that $T(X) \to  \bA_{k}^1$  is flat provided $X$ is the trivial deformation of $Y$ over $D$. 
\end{example}



\begin{example}\label{defnode}
Here is a more difficult ``non-example" to Theorem \ref{tanflat} where, unlike in Remark \ref{fibex}, we have a non-trivial deformation over the dual numbers. Let $k$ be a field. Consider the infinitesimal deformation 
\begin{equation*}
N_1:= \spec{k[x,y,t]/(xy-t, t^2)}
\end{equation*} of the node $N_0 = \spec{k[x,y]/(xy)}$ over the dual numbers $D=\spec{k[t]/(t^2)}$. 
Let  $a+bt$,  $c+dt$, and $e+ft$ be generic elements of $k[t]/(t^2)$. We substitute these arc variables into the defining equations of $N_1$ to obtain:
\begin{equation*}
\begin{split}
0 &=(a+bt)(c+dt) - (e+ft) = (ac-e)\cdot1 + (ad+bc - f) \cdot t \\
0 &= (e+ft)^2 = e^2\cdot 1 + 2ef\cdot t
\end{split}
\end{equation*}
This gives the equations defining $\underline{Hom}_{k}(D, N_1)$ as an affine scheme over $k$, and in the reduction, $e^2 = 0 \implies e =0$ so that the coordinate ring for $T(N_1) := \underline{Hom}_{k}(D, N_1)^{\red}$ is given by $k[a,b,c,d,f]/(ac, ad+bc - f)$.

It is clear that $T(N_1)$ is not irreducible, yet we will still be able to demonstrate that the morphism $\theta: T(N_1) \to  \spec{k[f]}$ induced from the flat morphism $N_1 \to D$ is again flat. In fact, we first will show that the smooth locus of this is given by $U = \theta^{-1}(\bA_{k}^1\setminus\{0\})$. 

To show this, let $A =k[f]$ and let $B = A[a,b,c,d]/I$ where $I = (F, G)$ with $F = ac$ and $G = ad+bc - f$, then it is enough to calculate the following determinate:
\[\mbox{Jac} := \begin{array}{|cc|}
\frac{\partial}{\partial a} F & \frac{\partial}{\partial c}F \\ 
\frac{\partial}{\partial a} G & \frac{\partial}{\partial c}G\end{array} = \begin{array}{|cc|}
c & a \\ 
d & b 
\end{array} = cb - ad \]
If $\mbox{Jac} = 0$, then $ad = bc $ and so using the equation $G=0$ gives $2ad = f = 2bc$, but either $a= 0$ or $c = 0$ and in either case then $f = 0 $. Therefore, $ f \neq 0$ implies $\mbox{Jac} \neq 0$ which implies the map $\theta$ is smooth away from points which map $f$ to $0$. 

Now, if $f = 0$, then the equation $G= 0$ implies $ad = - bc$, which implies $2ad = \mbox{Jac} = 2bc$.  For $\mbox{char}(k) = 2$, this immediately implies $\mbox{Jac} = 0$. Otherwise, we know either $a=0$ or $c=0$ (this is given by the equation $F =0$) and so in either case $\mbox{Jac} =0$. This together with the preceding paragraph show that the smooth locus of $\theta$ is given by  $U = \theta^{-1}(\bA_{k}^1\setminus\{0\})$, as claimed. 

Actually, one can quickly see that both of the irreducible components of $T(N_1)$ are isomorphic to $Q \times_k \bA_{k}^1$ where $Q = \spec{k[x,y,t]/(xy-t)}$. This is straightforward since on the component given by $a=0$ for example, the coordinate ring reduces to $k[b,c,d,f]/(bc-f)$ where $d$ is now a free variable. The reader may quickly check the other component. Regardless, since the map $\theta$ is smooth away from the fiber above $f = 0$, it has constant relative dimension equal to $2$ on each irreducible component when restricted to the smooth locus $U$. 

Now, notice that the central fiber given by  $f=0$  is actually the reduced tangent space of the node $N_0$:
\begin{equation*}
\theta^{-1}(0) \cong \underline{Hom}_{k}(D, N_0) \cong T(N_0)
\end{equation*}
as it is the spectrum of the ring $k[a,b,c,d]/(ac, ad+bc)$. Then, both irreducible components of $\theta^{-1}(0)$ are isomorphic to $N_1 \times_k \bA_{k}^1$, and so $\mbox{dim}(\theta^{-1}(0)) = 2$. 

In summary, we have explicitly shown that $\theta : T(N_1) \to \bA_{k}^{1} $ has constant relative dimension equal to $2$, and so, by the miracle of flatness  (Theorem 23.1 on page 179 of \cite{mat1987}), $\theta$ must be flat since the base is regular and $T(N_1)$ is a complete intersection and hence it is Cohen-Macaulay. 
\end{example}
\begin{remark} \label{domremark}
Examples \ref{fibex} and \ref{defnode} are not really non-examples because we can relax the condition in Theorem \ref{tanflat} to the case that $T(X)$ is reducible provided the restriction of $\theta : T(X) \to \bA_{k}^{1}$ to each of its irreducible components is a dominate morphism as noted Theorem \ref{LiuFlat}. This is clearly the case in Example \ref{defnode} since, as we noted above, both irreducible components are of the form $N_1 \times_k \bA_{k}^{1}$ and are mapping surjectively onto $\spec{k[f]}$.  In the case of Example \ref{fibex}, the statement is immediate. 
\end{remark}

\begin{example} \label{cmex}
Going back to Example \ref{cubicA}, we will check directly that it is flat in characteristic $2$ but by using either of the technique of \ref{defnode}. One may check that the Jacobian in this case is given by $\mbox{Jac} = a^4$. This then implies that the open subscheme $\theta^{-1}(\bA_k^1 \setminus \{0\})$ is contained in the smooth locus $U$ of $\theta$ (although, they do not agree in this case). Thus, away from $f = 0$, the morphism $\theta$ is of constant relative dimension $2$. 

Moreover, $\theta^{-1}(0)$ is the union of two subschemes $V$ and $W$ each of dimension $2$. Here, $V$ is the singular locus of $\theta$ given by $a = 0$ (the condition that $a=0$ implies $c = 0$ and $f = 0$), and thus $V$ is isomorphic to $\bA_{k}^2$. The subscheme $W$ is cut out by $b = f = 0$, and so $W$ is isomorphic to $C \times_k \bA_{k}^1$. 

Regardless, $\theta$ has fibers of constant dimension equal to $2$ and the base is regular. So, by the miracle of flatness, we just need to make sure that the domain is Cohen-Macaulay. The skeptical reader can run the Singular code provided in Remark \ref{cmcode} to truly check that this ring is Cohen-Macaulay. Interestingly, the morphism $\theta$ in Example \ref{cubicA} is flat in every characteristic, and, again, the underlying reason for this behavior has to do with the fact that each irreducible component is dominating $\bA_k^1$ via $\theta$. 
\end{example}

\begin{remark}
The cause of this degenerate behavior in Example \ref{cmex} is that the embedding of an affine truncated linear $n$-jet scheme into $\bA_k^M$  can be alternatively  described in terms of the universal derivation $\alpha \mapsto (j!\alpha_j^{(l)})$ for $\mbox{char}(k) = 0 $ or $\mbox{char}(k)> n$ (cf. page 5 of \cite{mus2001}). For more on how derivations relate to jet spaces, we refer the interested reader to \cite{voj2004}. For simplicity then we will later restrict to the case where the underlying field $k$ is of characteristic zero. 
\end{remark}

Considering Example \ref{defnode} and the previous remark, it would thus be interesting to know in more generality when the irreducible components dominate $\bA_{S}^{1}$ via the natural morphism $\theta: T(X) \to \bA_{S}^{1}$, whenever $S$ is a reduced Noetherian local ring in any characteristic. It would also seem to be an interesting yet difficult question to try push these types of questions to the non-reduced structure -- i.e., investigate when the morphism  from $\underline{Hom}_{S}(D, X)$ to  $\underline{Hom}_{S}(D,D)$ is flat. We will not directly concern ourselves with these general problems in this paper. If however we restrict ourselves to the reduced structure, then it is natural in this context to ask the following general question.

\begin{question}
Given an arbitrary $S \in \mathbf{Sch}$, let $X, Y,$ and $Z$ be objects of $\sch{S}$ whose structure morphisms are all locally of finite presentation. Assume further that $X$ is finite and flat over $S$ and that $Y$ and $Z$ are both $S$-separated. Given a flat morphism $f: Y \to Z$, when is it the case that the induced natural morphism 
$$\theta_f: \underline{Hom}_S(X,Y)^{\red} \to \underline{Hom}_S(X,Z)^{\red}$$
is also flat?
\end{question}

It is decidedly not true in general that the morphism $\theta_f$ is flat.  We specialize this question to the following simplified case. Let $S = \spec{k}$ where $k$ is a field, $X = \spec{k[t]/(t^2)}$, and  $f: Y \to \mathbb{A}_{k}^{1}$ be a flat family of curves, then it is not always the case that the induced natural morphism  $\theta_f$ on the reduced tangent spaces
$$\theta_f : T(Y) \to \mathbb{A}_{k}^{2}$$
is flat.

\begin{example} \label{nonflat}
Let $W = \spec{A}$ with $A = k[x,y,z]/(x^2-y^2z)$ be the flat family over $\mathbb{A}_k^1 = \spec{k[z]}$. 
Let $\wideparen{x} = a+bt,$ $\wideparen{y} = c+dt,$ and $\wideparen{z}= e+ft$ be arc variables where we compute $\wideparen{x}^2-\wideparen{y}^2\wideparen{z} = 0 $ in the reduction over $k[t]/(t^2)$.
This gives the ideal $$I = (a^2 - c^2e, 2ab - 2cde - c^2f) $$ in the ring $R = k[a,b,c,d,e,f]$ and an isomoprhism $ T(W) \cong \spec{R/I} $. 

If $\mbox{char}(k) = 2 $, then $I$ reduces to the ideal $(a^2 - c^2e, c^2f)$ and we note that $b$ and $d$ are a free variables now.  For $(e,f) = (0,0)$, it is easy to see that the fiber is equal to $\mathbb{A}_{k}^{3} = \spec{k[b,c,d]}$. The dimension of the base $\mathbb{A}_{k}^{2}$ at $(0,0)$ is clearly $2$, so we just need to find the dimension of $T(W)$ at a point in the preimage of $(0,0).$ 

Actually, at any point in the preimage, we have $f \neq 0 \implies c^2 = 0 \implies c = 0$. Then, $0 = a^2-c^2e = a^2$ implies $a = 0$. Thus, for example, the local ring at the maximal ideal defined by the origin is given by localizing $k[b,d,e,f]$ at $(b,d,e,f)$ which is clearly a local ring of Krull dimension $4$. Therefore, the relative dimension of $\theta_f$ at the origin is $2$, yet the dimension of the fiber is $3$. Thus,  $\theta_f$ cannot be flat. 
\end{example}

\begin{remark}\label{cmcode} The ring in $R/I$ in Example \ref{nonflat} is even a Cohen-Macaulay ring, which we can check using Singular \cite{DGPS}. We include this here for illustrative purposes and since the code below is referenced in Example \ref{cmex}.  Following SINGULAR Example 7.7.8 on page 426 of \cite{gre2002}, we will use the following code: 

\begin{lstlisting}[language = C]
LIB "homolog.lib";  //Loads library with homological algebra procedures

//Procedure which calculates depth of a module
proc depth(module M) {ideal m=maxideal(1); int n=size(m); int i;
	while(i<n) {i++;
	if(size(KoszulHomology(m,M,i))==0){return(n-i+1);} } return(0); }			
	
//Procedure tests whether module is CM: returns 1 if true. 	
proc CohenMacaulayTest(module M) {return(depth(M)==dim(std(Ann(M)))); }
	
//Define module and test for CM	
ring R = 2, (a, b, c, d, e, f), dp;		
ideal I = a2 - c2e, c2f;			
module M = I*freemodule(1);			
CohenMacaulayTest(M);				
\end{lstlisting}

The output of this code is $1$, which demonstrates that $R/I$ is Cohen-Macaulay. This is not at  all surprising since $R$ is automatically a complete intersection by Proposition 1.4 on page 7 of \cite{mus2001}, and hence Cohen-Macaulay. 
\end{remark}

\begin{remark}
As far as checking the ring is Cohen-Macaulay in characteristic $2$, it is enough just to work over the finite field $\mathbb{F}_2$, because, by base change (cf., \cite[\href{https://stacks.math.columbia.edu/tag/045P}{Tag 045P}]{stacks-project}), this will also prove that $R/I$ is Cohen-Macaulay over any field $k$ of characteristic equal to $2$.
\end{remark}

\begin{example}
Let us then consider the case of Example \ref{nonflat}, but this time we assume $k = \bQ$.  Actually, we can use Singular to check for flatness directly in this case by asking it to compute the first torsion module. First note, that
\begin{equation*}
\mbox{Tor}_1^{\bQ[e,f]}(\bQ,R/I) \cong \mbox{Tor}_1^{R}(\bQ[a,b,c,d],R/I).
\end{equation*}

\begin{lstlisting}[language=C]
LIB "homolog.lib"; 

ring R = 0, (a,b,c,d,e,f), dp;
ideal I = a2-c2e, 2ab - 2cde - c2f;
matrix Ph[1][2] = I;
matrix Ps[1][2] = e, f;
Tor(1,Ps,Ph);
\end{lstlisting}

The output of this code is 
\begin{verbatim}
// dimension of Tor_1:  3

_[1]=gen(1)
_[2]=gen(2)
_[3]=gen(3)
_[4]=gen(4)
_[5]=f*gen(5)
_[6]=e*gen(5)
_[7]=a*gen(5)
\end{verbatim}
\end{example}

\section{General Auto-Arc Spaces}\label{sec2}

\begin{definition}
Let $i: X \to S$  be an object of $\sch{S}$ and $Z$ an object of $\fatpoints{S}$ where $S$ is an arbitrary scheme. We say that a scheme $Y$ is an {\it infinitesimal deformation over} $Z$ of $X$ if there is a commutative diagram

\[\begin{tikzcd}
X \arrow[d,"{i}"] \arrow[r, hook] & Y\arrow[d, "{j}"]\\
S  \arrow[r] &Z
\end{tikzcd}\]
such that the structure morphism $j : Y \to Z$ is flat and induces an isomorphism $X \cong Y\times_Z S$.
\end{definition}

Given any infinitesimal deformation $Y$ over $Z$  such that $Y$ is separated and locally of finite presentation over $Z$, we define
\begin{equation*}
\sA_Z(Y) := \underline{Hom}_S(Z,Y)^{red}
\end{equation*}
and call it {\it the auto-arc space of} $Y$ {\it with respect to} $Z$, and in the special case that $Y\cong Z$, we call $\sA_Z := \sA_Z(Z)$ {\it the auto-arc space of} $Z$. We therefore have the natural induced morphism
\begin{equation*}
\theta_j : \sA_Z(Y) \to \sA_Z
\end{equation*}

Assume now that $S$ is a reduced scheme, then we have an induced isomorphism
$$\underline{Hom}_S(Z,X)^{\red} \cong \underline{Hom}_S(Z,Y) \times_Z S.$$

\begin{remark}
Note that the morphism $\underline{Hom}_S(Z,Y) \to Z$ is very often non-flat. For instance, if $S = \spec{k}$ with $\mbox{char}(k) \neq 2$, and if $Y$ and $Z$ are both the dual numbers over $S$, then on the level of coordinate rings, the morphism of schemes above induces a ring monomorphism $k[t]/(t^2) \to k[t,s]/(t^2, 2st)$ defined by $t \mapsto t$, which is clearly not flat. \end{remark}

In the category $\sch{S}$, any infinitesimal deformation $Y$ over $Z$ gives rise to the following induced commutative diagram 
\[\begin{tikzcd}
\underline{Hom}_S(Z, X)\arrow[d]  \arrow[r, hook]&\underline{Hom}_S(Z,Y) \arrow[d] \arrow[r, " "] &\underline{Hom}_S(Z,Z)\arrow[d] \\
X\arrow[d,"i"]  \arrow[r, hook] &Y\arrow[d,"j"]  \arrow[r, "j"] & Z  \\
S \arrow[r] & Z  & \
\end{tikzcd}\]
provided $Y$ is separated and locally of finite presentation over $Z$.

If we further assume $S$ and $X$ are reduced schemes, then by taking the fiber product with respect to $S \to Z$, we obtain the following commutative diagram 

\begin{equation}\label{maindiag}
\begin{tikzcd}
 \underline{Hom}_S(Z,X)^{\red}  \arrow[d]  \arrow[r, hook] &\sA_Z(Y) \arrow[d] \arrow[r,"{\theta_j}"] &\sA_Z \arrow[d] \\
X \arrow[r, "{\cong}"] &X \arrow[r, ""] & S
\end{tikzcd}
\end{equation}

\begin{subsection}{The case at the germ of a plane curve singularity}
One special case of the above is 
given when we assume that  $S = \spec{k}$ where $k$ is an algebraically closed field. Given a closed point $x$ of $X$, we let $Z$ be the $n${\it-th order jet} $J_p^nX := \spec{\sO_{X,x}/\mathfrak{m}_x^{n+1}}$ which is an object in $\fatpoints{k}$. In which case, we have the induced closed immersion
\begin{equation*}
\underline{Hom}_S(Z,Z) \hookrightarrow \underline{Hom}_S(Z,X) \times_S Z
\end{equation*}
which extends the diagram above to  the commutative diagram 
\[\begin{tikzcd}
\underline{Hom}_S(Z,Z) \arrow[d] \arrow[r, hook] &\underline{Hom}_S(Z, X)\arrow[d]  \arrow[r, hook]&\underline{Hom}_S(Z,Y) \arrow[d] \arrow[r, " "] &\underline{Hom}_S(Z,Z)\arrow[d] \\
Z \arrow[r, hook] &X\arrow[d,"i"]  \arrow[r, hook] &Y\arrow[d,"j"]  \arrow[r, "j"] & Z \\
&S \arrow[r] &Z & \
\end{tikzcd}\]
Then, by taking the fiber product with respect to $S \to Z$, we obtain the following commutative diagram 

\[\begin{tikzcd}
\sA_Z \arrow[d] \arrow[r,hook] & \underline{Hom}_S(Z,X)^{\red}  \arrow[d]  \arrow[r, hook] &\sA_Z(Y) \arrow[d] \arrow[r,"{\theta_j}"] &\sA_Z \arrow[d] \\
S \arrow[r, hook] &X \arrow[r, "{\cong}"] &X \arrow[r, ""] & S
\end{tikzcd}\]

Thus, $\sA_Z$ is a closed subscheme of the fiber above the point given by $S \to X$, or, in other words, we have a closed immersion
\begin{equation}\label{isomor}
\sA_Z \hookrightarrow \underline{Hom}_S(Z,X)^{\red} \times_X S
\end{equation}

\begin{proposition}\label{yay}
Let $S = \spec{k}$ where $k$ is an algebraically closed field. Let $X$ be a reduced, separated $S$-scheme, locally of finite type over $S$. Then, the closed immersion \ref{isomor} above is an isomorphism. 
\end{proposition}

\begin{proof}
This is essentially a restatement of Lemma 3.2 of \cite{auto}. The only difference being that finite type there is being replaced with locally of finite type here. 
\end{proof}

\begin{example}
Let $X$ be a smooth curve over an algebraically closed field $k$ and let $Y$ be the trivial deformation over the dual numbers $D = J_p^1X$. Then, $\sA_D(Y) \cong  \underline{Hom}_S(D,X)^{\red}  \times_k \sA_D$.  Note that these spaces are all tangent bundles as studied in Section \ref{sec1}. Thus, we have the following commutative diagram 
\[\begin{tikzcd}
\bA_k^1 \arrow[d] \arrow[r,hook] &T(X) \arrow[d]  \arrow[r, hook] &T(X) \times_k \bA_k^1 \arrow[d] \arrow[r,"{\theta_j}"] &\bA_k^1 \arrow[d] \\
\spec{k} \arrow[r, hook] &X \arrow[r, "{\cong}"] &X \arrow[r, ""] & \spec{k}
\end{tikzcd}\]
where $\theta_j$ is projection onto the second factor. Note that $T(X)$ is \'{e}tale locally isomorphic to $X \times_k \bA_k^1$ in this case since $X$ is smooth. 
\end{example}

\begin{example} \label{planearc}
Let $k$ be an algebraically closed field with $\mbox{char}(k) \neq 2, \ 3$. Let $X$ be the cuspidal cubic $\spec{k[x,y]/(y^2 -x^3)}$ and let $p$ be the singular point given by the origin. Let $Z = J_p^{n-1}X$ and let $\sL_m(X) =  \underline{Hom}_S(J_p^mX,Y)^{\red} $ denote the reduced truncated linear arc space. It is proven in \cite{sto2017} that 
$\sA_Z \cong \sL_{m}(X) \times_k \bA_k^7$
where $m = 2(n-3)$ whenever $n \geq 4$.  Thus, in this case, for an infinitesimal deformation $Y$ of $X$ over $Z$, we have
\[\begin{tikzcd}
\sL_{m}(X) \times_k \bA_k^7 \arrow[d] \arrow[r,hook] & \underline{Hom}_S(Z, X)^{\red} \arrow[d]  \arrow[r, hook] & \sA_Z(Y) \arrow[d] \arrow[r,"{\theta_j}"] &\sL_{m}(X) \times_k \bA_k^7 \arrow[d] \\
\spec{k} \arrow[r, hook] &X \arrow[r, "{\cong}"] &X \arrow[r, ""] & \spec{k}
\end{tikzcd}\]

\end{example}

\begin{remark}\label{help}
It is worth noting that the relationship between $\sL_m(X)$ and the auto-arc space $\sA_Z$ noticed in Example \ref{planearc} has been generalized to all plane curve singularities $(X,p)$. This is worked out in significant detail in \cite{auto}. \end{remark}

\begin{question}
Given a plane curve singularity $(X,p)$, it should be somewhat straightforward to extend the above diagram to the so-called {\it mixed auto-arc spaces} $\underline{Hom}_S(J_p^nX, J_p^mX)^{\red}$ for $n, \ m$ sufficiently large. More generally, it should be possible to obtain ``closed expressions" related to the truncated linear arc spaces (say in the Grothendieck ring of varieties) of $\underline{Hom}_S(Z, Z')^{\red}$ whenever $Z$ and $Z'$ are jets of different plane curve singularities. For higher embedding dimension, say for germs of surface singularities, computations carried out by the author clearly show a ``recursive pattern" for $\sA_Z$ in the Grothendieck ring of varieties. Moreover, they seem to be related to iterated linear jet spaces, but establishing a direct relationship in analogy to the situation for plane curves as mentioned in Remark \ref{help} remains somewhat elusive. 
\end{question}

\end{subsection}

\begin{subsection}{Further general statements for auto-arc spaces.}

We go back to assuming that $X$ is an arbitrary reduced scheme over $S=\spec{k}$ where $k$ is an algebraically closed field. We let $Z$ be an arbitrary object of $\fatpoints{k}$ and we assume that $Y$ is a local deformation of $X$ over $Z$ which is also separated and locally of finite type over $Z$. Then, as noted at the beginning of Section \ref{sec2}, we have the following commutative diagram
\[\begin{tikzcd}
 \underline{Hom}_S(Z,X)^{\red}  \arrow[d]  \arrow[r, hook] &\sA_Z(Y) \arrow[d,"{\pi}"] \arrow[r,"{\theta}"] &\sA_Z \arrow[d] \\
X \arrow[r, "{\cong}"] &X \arrow[r, ""] & S
\end{tikzcd}\]

Of course, this immediately implies that we have a natural induced morphism
\begin{equation*}
\sA_Z(Y) \xrightarrow{\pi\times\theta} X \times_S\sA_Z
\end{equation*}
which is locally a piecewise trivial fibration with affine fiber whenever $X$ is smooth, and in fact we have the following lemma. 

\begin{lemma}\label{etlem}
Given our basic assumptions of this particular subsection,  let $X_{\sm}$ denote the smooth locus of $X$. Then, the morphism $\pi\times \theta$ above is a piecewise trivial fibration with affine fibers away from the singular locus $X\setminus X_{\sm}$. Moreover, 
$\underline{Hom}_S(Z,X_{\sm})^{\red} \times_k \sA_Z$ and $\pi^{-1}(X_{\sm})$ are \'{e}tale locally isomorphic. 
\end{lemma}
\begin{proof}
We can find a covering by affine opens $U_{\alpha} \subset X_{\sm}$ with \'{e}tale morphisms $U_{\alpha} \to \bA_k^{r_{\alpha}}$. We therefore have open affines $V_{\alpha} \cong U_{\alpha}\times_k Z$ given by the embedded deformation of $U_{\alpha} \subset X$, which will 
allow us to cover $\pi^{-1}(X_{\sm})$  by $V_{\alpha}$ with \'{e}tale morphisms $V_{\alpha} \to \bA_k^{r_{\alpha}}\times_k Z$. Thus, away from the singular locus 
\begin{equation*}
\sA_Z(Y)\times_X V_{\alpha} \cong \sA_Z(V_{\alpha}) \cong \underline{Hom}_S(Z, U_{\alpha})^{\red} \times_k \sA_Z
\end{equation*}
\end{proof}

\begin{remark} \label{oppmap}
In general, the closed immersion $\iota : \underline{Hom}_S(Z, X)^{\red} \into \sA_Z(Y)$  admits a section $s$ such that $s\circ \iota$ is the identity morphism, which then by the universal property of fiber products yields an induced natural morphism $$\sA_Z(Y) \to \underline{Hom}_S(Z, X)^{\red} \times_k \sA_Z$$
which, by Lemma \ref{etlem}, is \'{e}tale locally an isomorphism away from the singular locus $X\setminus X_{\sm}$. 
Actually, the section $s$ is induced by applying the reduction functor to the morphism 
$$\underline{Hom}_S(Z,Y) \to \underline{Hom}_S(Z,X)$$
which is induced from restricting the image of a morphism $Z \to Y$ to the image of the closed immersion $X \into Y$. 
\end{remark}

\begin{lemma} \label{basiclem}
Let $O \in \sA_Z$ correspond to the trivial endomorphism of $Z$ and assume our basic assumptions of this particular subsection. Then, 
\begin{equation*}
 \underline{Hom}_S(Z,X)^{\red} \cong \theta^{-1}(O)
\end{equation*}
\end{lemma}

\begin{proof}
Let $Z=\spec{R}$ where $(R, \mathfrak{m})$ is a finitely generated local Aritinian $k$-algebra, and let $O$ denote the point on $\sA_Z$ which corresponds to the ring endomorphism $\varphi_O$ of $R$ given by the map onto the residue field $k$ (i.e., $\varphi_O$ is the composition $R \to R/\mathfrak{m} \into R$). 
Then, this lemma is clearly true since the problem is local (i.e., we may assume $X$ is an affine scheme) and therefore we may reduce to case where a point on $\underline{Hom}_S(Z,X)^{\red}$ gives rise directly to a morphism $h: Z \to X$ which fits into the commutative diagram
\[\begin{tikzcd}
Z \arrow[r,"h"] \arrow[rd] & X \arrow[d,"{i}"] \arrow[r, hook] & Y\arrow[d, "{j}"]\\
 \ & S  \arrow[r] &Z
\end{tikzcd}\]
Thus, post-composing with the closed immersion $X\into Y$ gives a morphism $\bar{h} : Z \to Y$ and therefore corresponds to a point on $\sA_Z(Y)$, which means that the point on $\sA_Z$ given by $\theta(\bar{h})$ must be equivalent to the morphism $Z \to S \to Z$ which is precisely the morphism $\spec{\varphi_O}$ where $\varphi_O$ is the ring endomorphism discussed above. 
\end{proof}

\end{subsection}
\begin{subsection}{The situation for locally complete intersections.}

An object $X$ of $\sch{k}$ which is reduced, separated and finite type over $k$ is called a {\it variety over} $k$. We have the following well-known result concerning deformations of locally complete intersection varieties. We need the following fact.

\begin{proposition}\label{deformlci}
Let $Z \in \fatpoints{k}$ with $k$ a field. Moreover, let $X$ be  locally complete intersection variety over $k$.  Then, any deformation $Y$ over $Z$ is locally complete intersection. 
\end{proposition}

\begin{proof}  This is Theorem 9.2 on page 74 of \cite{hart2009}.
\end{proof}

Assume in the rest of this subsection that $X$ is a variety of dimension $d$. Let $\delta_Z := \dim{\ \sA_Z}$ and $\ell:=\ell(Z)$ be the length of $Z$. In general, we have $$\dim{\ \underline{Hom}_S(Z,X)} \geq d\ell$$ so that 
\begin{equation} \label{inediteq1}
\dim{\ \sA_Z(Y)} \geq d\ell +\delta_Z
\end{equation} 
by Lemma \ref{etlem} and Remark \ref{oppmap}.  
Now, following the decomposition of Proposition 1.4 on page 7 of \cite{mus2001}, we consider 
\begin{equation} \label{dimauto}
\sA_Z(Y) = \pi^{-1}(X_{\sing}) \cup\overline{\pi^{-1}(X_{\sm})}
\end{equation}
where $X_{\sing}$ denotes the singular locus of $X$. Therefore, if we assume moreover that $\sA_Z(Y)$ is pure dimensional, then we obviously have
\begin{equation} \label{bediteq1}
\dim{\ \sA_Z(Y)} = d\ell +\delta_Z
\end{equation}
and if $\sA_Z(Y)$ is also irreducible, then 
\begin{equation*}
\dim{\  \pi^{-1}(X_{\sing})} <  d\ell+\delta_Z
\end{equation*}

Notice that Inequality \ref{inediteq1} then clearly implies that 
\begin{equation}
\frac{\delta_{Z}}{\ell} \leq \frac{\dim{\ \sA_Z(Y)}}{\ell} - d
\end{equation}
For any closed germ $(Z,P)$ giving rise to a limit of fat points $Z_n$ and a sequence of deformations $Y_n$ over $Z_n$ of $X$, we define
\begin{equation}
\begin{split}
\delta_{(Z,P)}^{*}(X) &:= \lim_{n\to \infty}  \frac{\dim{\ \sA_{Z_n}(Y_n)}}{\ell(Z_n)} - d \\
e(Z,P) &: = \lim_{n\to \infty}  \frac{\dim{\ \sA_{Z_n}}}{\ell(Z_n)} 
\end{split}
\end{equation}
where the former is a slight variant of the {\it asymptotic defect} defined in Equation 28 on page 26 of \cite{sch2} and the later is defined without change on page 27 of \cite{sch2}. We note that we then have the inequality 
\begin{equation}
e(Z,P) \leq \delta_{(Z,P)}^{*}(X)
\end{equation}
\begin{remark}
We are suppressing notation here.  More explicitly, the asymptotic defect defined above is heavily dependent on the choice of deformations. Regardless, our goal here is just to simply produce the lower bound $e(Z,P)$.  
\end{remark}

Thus, when the limits above exist, we can consider the so called {\it regulated defect} given by $\delta_{(Z,P)}^{*}(X)/e(Z,P)$. More generally, we define
\begin{equation}
R_{(Z,P)}^{\sY}(X) :=   \limsup_{n}\{\frac{ \dim \ \sA_{Z_n}(Y) -d\ell(Z_n)}{\dim \ \sA_{Z_n}}\}
\end{equation}
and we call it the {\it regulated defect of} $X$ {\it at} $(Z,p)$ {\it along the formal deformation} $\sY=\varprojlim Y_n$ where $Y_n$ is an infinitesimal deformations of $X$ over $Z_n$. 

Regardless, we continue by adapting the proof of Proposition 1.4 on page 7 of \cite{mus2001} to the current situation above, and so we now assume that $\dim{\ \sA_Z(Y)} = d\ell +\delta_Z$ where $d$ is the dimension of $X$ and $\ell$ is the length of the fat point $Z$. Assume now that $X$ is a locally complete intersection variety over $k$ and let $Y$ be a deformation of $X$ over a fat point $Z$ so that, by Theorem \ref{deformlci}, $Y$ is also a locally complete intersection variety over $Z$.  The problem is local so we can reduce to the case where $X$ and $Y$ are affine, and, of course, we choose these affines so that $Y$ is a complete intersection. Therefore, by assumption $Y \subset \bA_{Z}^N$ is defined by $N-d$ equations. 

We therefore have that $\sA_Z(Y)$ is defined by $(N-d)\ell$ equations as subvariety of  $\sA_{Z}(\bA_Z^N) \cong \bA_{k}^{N\ell}\times_k\sA_Z$. Therefore, any irreducible component of $\sA_Z(Y)$ has dimension  greater than or equal to $N\ell +\delta_Z - (N-d)\ell = d\ell+\delta_Z$. But, then by assumption on the dimension of $\sA_Z(Y)$, it must be of pure dimension and a complete intersection.  We therefore have the following result. 

\begin{theorem} \label{lci}
Let $X$ be a locally complete intersection variety over a field $k$ and let $Z$ be an object of $\fatpoints{k}$. Let $Y$ be a deformation of $X$ over $Z$ and assume that $\dim \ \sA_Z(Y) = d\ell +\delta_Z$.  Then, $\sA_Z(Y)$ must be of pure dimension and  a locally complete intersection over $k$. 
\end{theorem}

\begin{remark}
We note that Example \ref{defnode} shows that the closed set $\overline{\pi^{-1}(X_{\sm})}$ of $\sA_Z(Y)$ is not necessarily irreducible, and so it will frequently be the case that the dimension of $ \pi^{-1}(X_{\sing})$ is strictly less than  $d\ell +\delta_Z$,  yet $\sA_Z(Y)$ will be reducible. This is in stark contrast with the classical truncated linear arc case as was also pointed out in Remark \ref{redcent} and Remark \ref{bigremark}.
\end{remark}

We will need Theorem \ref{lci} in the proceeding section when we investigate the linear jet case, but it also has another important implication, which we state below. 

\begin{theorem} 
Assume that $X$ is a locally complete intersection variety over an algebraically closed field $k$. Let $Z$ be the object of $\fatpoints{k}$ given by the $n$th jet scheme $J_p^nX$ at some closed point $p$ of $X$. Then, the auto-arc space $\sA_Z$ of $Z$ is a locally complete intersection variety over $k$ whenever $\dim \ \underline{Hom}_S(Z,X)^{\red} = d\ell$. 
\end{theorem}

\begin{proof}
The key elements of the proof do not change if we apply the adapted argument above to $\underline{Hom}_S(Z,X)^{\red}$ as opposed to $\sA_Z(Y)$. Here, one only needs to assume that the dimension of this generalized arc space is equal to $d\ell$. Now, we may use Proposition \ref{yay}, which states that the space $\sA_Z$  is obtained by a flat base change of $\underline{Hom}_S(Z,X)^{\red} \to X$, which is a locally complete intersection morphism by \cite[\href{https://stacks.math.columbia.edu/tag/09RL}{Tag 09RL}]{stacks-project}. This then implies $\sA_Z$ is a locally complete intersection variety over $k$ by \cite[\href{https://stacks.math.columbia.edu/tag/069I}{Tag 069I}]{stacks-project}. 
\end{proof}

\begin{example} \label{dimun}
As we noted in Example \ref{planearc}, $\sA_Z \cong \sL_{2(n-3)}(X)\times_k\bA_k^7$ when $Z$ is the $(n-1)$th jet of the cuspidal cubic $X$ given $y^2=x^3$ at the origin $O$ for $n \geq 4$. In this case, not only is it clearly seen that $\sA_Z$ is a locally complete intersection (by Proposition 1.4 on page 7 of \cite{mus2001}), but it is also reducible directly by Corollary 4.2 on page 19 of \cite{mus2001}. Note also that $\delta_{(Z,P)}^{*}(X) \geq  e(X,O) = \lim_{n\to\infty} \frac{2n-6+7}{2n-1} = 1$.  As we noted earlier, this behavior generalizes to all plane curve singularities as discussed in \cite{auto}.  \end{example}

\begin{question}
Let $X$ be a locally complete intersection variety  of pure dimension $d$  over an algebraically closed field $k$ and assume that $\dim \ \underline{Hom}_S(Z,X)^{\red} = d\ell$. Let $Z$ be the object $J_p^nW$ of $\fatpoints{k}$ where $W$ is a locally complete intersection variety over $k$. Given a deformation $Y$ of $X$ over $Z$ such that $\dim \ \sA_Z(Y)  = d\ell +\delta_Z$, what is the flat locus of the induced morphism $\theta: \sA_Z(Y) \to \sA_Z$?
\end{question}

\begin{remark}
For any deformation $Y$ of $X$ over $Z$ such that $\dim \ \sA_Z(Y)  = d\ell +\delta_Z$, one can show that the induced morphism $\theta$ is flat by the miracle of flatness (cf. Theorem 23.1 on page 178 of \cite{mat1987}) provided  $\sA_Z$ is regular. 
\end{remark}

\end{subsection}

\section{The situation over linear jets.} \label{sec3}

Now, we will study this problem over the linear jets $Z_n = \spec{k[t]/(t^{n+1})}$. We let $X$ be a variety over $k$. 
We assume that there is a deformation $X_n \to Z_n$. For example, this occurs when $X$ is a complete intersection subvariety of $X=\mathbb{P}_{k}^{N}$ or more generally $X$ is a locally complete intersection in $X$ and the obstruction in $H^1(\sN_X)$ vanishes (cf. Theorem 9.2 on page 74 of \cite{hart2009}). Note then that $\sA_{Z_n}(X_n)$ is the same as the reduced truncated linear arc space $\sL_n(X_n)$.  

\begin{remark}
The auto-arc spaces in this context, hereafter always denoted by $\sL_n(X_n)$, are similar to the truncations $Gr_n(X_n)$ as studied in \cite{loe2003} and more recently in \cite{nic2011} provided the underlying field $k$ has equal characteristic. These later spaces are truncated versions of an infinite arc space, therein denoted by $Gr(\sX)$, where $\sX$ is a smooth formal scheme. \end{remark}

\begin{lemma}
Let $S= \spec{A}$ where $A$ is a reduced Noetherian local ring with residue field $k$. Let $Z_n$ be the object of $\fatpoints{S}$ given by $S \times \spec{\bZ[t]/(t^{n+1})}$. Then, $$\sA_{Z_n} \cong \bA_{S}^n$$
\end{lemma}

\begin{proof} 
This is the first part of Lemma 4.3 on page 141 of \cite{sto2019}.
\end{proof}

Of course then $\sL_n(Z_n) = \sA_{Z_n} \cong \bA_{k}^{n}$ where $Z_n=\spec{k[t]/(t^{n+1})}$ and $k$ is a field.  Thus, the morphism introduced in the beginning of Section \ref{sec2} is of the form
\begin{equation*}
\theta_n : \sL_n(X_n) \to \bA_k^{n}
\end{equation*}

\begin{theorem} \label{main}
Let $X$ be a locally complete intersection variety over an algebraically closed field $k$ of dimension $d$. 
Let $X_n$ be a deformation over $k[t]/(t^{n+1})$ such that $\dim \ \sL_n(X_n) = d(n+1) + n$. Then, the natural induced morphism $\theta_n :  \sL_n(X_n) \to \bA_k^{n}$ is flat. 
\end{theorem}

\begin{proof}
The fibers of $\theta_n$ are clearly seen to be constant. Then, by using the miracle of flatness (Theorem 23.1 on page 178 of \cite{mat1987}), we know that $\theta_n$ is flat if and only if $\sL_n(X_n)$ is a Cohen-Macaulay ring. This is true since under our assumptions $\sL_n(X_n)$ is a  locally complete intersection by Theorem \ref{lci}. 
 \end{proof}

\begin{corollary} \label{hilbcor}
Given the conditions of Theorem \ref{main}, and moreover, assume that $X \subset \bA_{k}^{N}$. Then, the fibers of $\theta_n$ give rise to a flat morphism $$\tilde\theta_n: \bA_k^{n} \to \underline{Hilb}(\bA_{k}^M)$$
where $M = (n+1)\cdot N$ and the image at the origin $O$ given by $\tilde\theta_n(O)$ is $\sL_n(X)$. 
\end{corollary}

\begin{proof} This is immediate. \end{proof}

\section{Remarks on linear auto-arcs for curves.}\label{sec4}

In this section, we briefly study the linear auto-arcs for deformations of curves. Even in this case, the situation is highly non-trivial. 

\begin{proposition} \label{curvecor}
Let $C$ be a curve over a field $k$, and let $C_1$ be a deformation over $\spec{k[t]/(t^2)}$ such that $T(C_1)$ is irreducible and of pure dimension. Then, the inverse image $\pi^{-1}(C_{\sing})$ of the singular locus under the morphism $\pi: T(C_1) \to C$ is always a subvariety of the fiber of the morphism $\theta_2 : T(C_1) \to \bA_k^1$ at the origin.  
\end{proposition}

\begin{proof}
When $C$ is non-singular the statement is trivial. Therefore, assume that $C$ is singular and let $V:= \pi^{-1}(C_{\sing})$. The decomposition in Equation \ref{dimauto} implies $\dim \ V \leq 2$ since $T(C_1)$ is irreducible and of pure dimension. Consider the restriction $\theta_1|_V : V \to \bA_k^1$ and assume for the sake of contradiction that it is non-constant. This implies that $\theta_1|_V(V)$ is a dense subset of $\bA_k^{1}$, from whence it follows from Theorem \ref{LiuFlat} that the restriction $\theta_1|_V$ is also flat. But, then the central fiber $V_0$ of $\theta_1|_V:V\to \bA_k^1$ over the origin is such that $\dim \ V_0 = \dim \ V  - 1$. Thus, $\dim \ V_0  \leq 1 $. 

By Lemma \ref{basiclem}, the central fiber of $\theta_2 : C_1 \to \mathbb{A}_k^1$ is isomorphic to $T(C)$, which by Commutative Diagram \ref{maindiag} implies that $V_0$ is the inverse image over the singular locus of the natural morphism $T(C) \to C$. But, at any singular point $p \in C$, the fiber under $T(C) \to C$, which is just the tangent space $T_p(C)$, has dimension strictly greater than $1$, which is a contradiction. Thus, $\theta_1|_V$ is a constant morphism and thus $V_0 = V$, which proves the claim. 
\end{proof}
\begin{remark} \label{bigremark}
Example \ref{defnode} shows that although the conditions for Proposition \ref{curvecor} are sufficient, they are not necessary - i.e., the reduced tangent bundle of the versal first order deformation of a node (given by $xy-t = 0$  with $t^2 =0$) is reducible, yet the central fiber of $\theta$ still contains the inverse image of the singular locus. 
\end{remark}

\begin{example}\label{dnode3}
We lift the deformation of Example \ref{defnode} to the second order deformation $N_2$ defined by 
\begin{equation*}
N_2: = \spec{k[x,y,t]/(xy-t^2-t, t^3)}
\end{equation*}
we create arc variables 
\begin{equation*}
\wideparen{x} = a_{11} + a_{12}t + a_{13}t^2 , \quad 
\wideparen{y} = a_{21} + a_{22}t + a_{23}t^2 ,\quad
\wideparen{t} = e + ft+gt^2
\end{equation*} 
 We note that $\wideparen{t}^3  = 0$ implies that $e = 0$ and places no further restrictions on $f$ and $g$, and so without loss of generality we may assume $\wideparen{t} = ft+gt^2$. Now, performing the remaining substitution, we have the arc equation $\wideparen{x}\wideparen{y}-\wideparen{t}^2-\wideparen{t}=0$ which generates the following list of equations defining $\sL_3(N_3)$. 
 \begin{equation*}
 \begin{split}
 a_{11}a_{21}  &= 0 \\
 a_{11}a_{22} + a_{12}a_{21} - f &= 0\\
 a_{11}a_{23}+a_{21}a_{13}+a_{12}a_{22} - g - f^2  &= 0 
 \end{split}
 \end{equation*}
 We note that the fiber over the origin $O$ of the natural morphism $\pi_2 : \sL_2(N_2) \to N$ is cut out by $a_{11} = a_{21} = 0 $.  This then implies that on this fiber, the variables $a_{13}$ and $a_{23}$ are free and $f=0$ . Thus, as a subvariety of $\sL_2(N_2)$, the fiber over the singular point in $N$ is given by $$\pi_2^{-1}(O) \cong \spec{k[x,y,g]/(xy - g)}\times_k \bA_{k}^2$$
 We note that $\dim \ \pi_2^{-1}(O) = 4$ and an irreducible component. Also, we note that the fiber over the singular locus leaves the central fiber of $\theta_2 : \sL_2(N_2) \to \bA_k^2$. 
\end{example}

\begin{remark} \label{rm3}
In light of Proposition \ref{curvecor}  and Example \ref{dnode3}, we consider the fiber of the singular locus for an irreducible curve such that $\sL_2(X_2)$ is irreducible and of pure dimension. Let $V:=\pi_2^{-1}(X_{\sing})$. By assumption, $\dim \ V < \dim \ \sL_2(X_2) = 5$. Assume $\theta_2|_V: V \to \bA_k^2$ is not constant, and let $L$ be a line on $\bA_k^2$ passing through the origin and contained in the image of $\theta_2|V$. By a change of coordinates if necessary, we can find a surjective map $\bA_k^2 \to L$ and consider the composition with $\theta_2|_{V}$ whose image will be dense and hence the composition will be flat. Therefore, the fiber at the origin of this composition, say $V_0$ will have dimension strictly less than $4.$ 

By Lemma 4.1 on page 18 of \cite{mus2001},  the dimension of any fiber over the singular locus is equal to $3$ or more. Thus, the dimension of $V_0$ must be exactly $3$ in this case. Although it may not be contained in the central fiber of $\theta$, we do have a picture for what the expected dimension should be. 
\end{remark}

We attempt to generalize the behavior noticed in Proposition \ref{curvecor} and Remark \ref{rm3}. For this, we let $\sA_Z^{*}$ denote the open subscheme of $\sA_Z$ isomorphic to $\underline{Aut}_S(Z)^{\red}$ for $Z$ an object in $\fatpoints{S}$, and we let $B_Z$ denote the compliment $\sA_Z \setminus \sA_Z^*$. Considering the behavior above and in that of Corollary \ref{curvecor}, we make the following definition.

\begin{definition}
Let $X$ be a scheme over another scheme $S$ and let $Z$ be an object in $\fatpoints{S}$. Let $Y$ be a deformation of $X$ over $Z$. Consider the natural induced morphism $\pi : \sA_Z(Y) \to X$ and let $V = \pi^{-1}(X_{sing})$. 
We say that the deformation $Y \to Z$ is {\it strong} if $V$ is contained in $\theta^{-1}(B_Z)$ and otherwise we call the deformation  {\it weak}. We say that the deformation is {\it very strong} if $V \subset \theta^{-1}(O)$ where $O$ is the point given by the trivial endomorphism of $Z$. 
\end{definition}

\begin{remark} In the case of deformations over $k[t]/(t^{n+1})$ and the corresponding truncated linear arc spaces, the notions of strong and very strong are equivalent, and for this reason, we will always refer to a very strong deformation as merely strong in this case. 
\end{remark}

\begin{remark}
We can see that the first order deformation in Example \ref{defnode} is a strong deformation, yet its second order cousin found in Example \ref{dnode3} is a weak deformation. 
\end{remark}

In general, for a strong deformation, the induced morphism $\pi\times \theta : \sA_Z(Y) \to X \times_S\sA_Z$ is then a piecewise trivial fibration away from the $\theta^{-1}(\sA_S^*)$ over the base $X_{\sm} \times_S \sA_Z^{*}$. Thus, for a strong deformation for example, we have a commutative diagram
\[\begin{tikzcd}
\sA_Z(Y)\setminus \theta^{-1}(B_Z) \arrow[d,] \arrow[r,hook] & \sA_Z(Y) \arrow[d,"{\pi|_{V}\times\theta}" left]  \arrow[r] \arrow[rd, "{\pi\times\theta}" description ]&\underline{Hom}_S(Z,X)^{\red}\times_S \sA_Z  \arrow[d] \\
X_{\sm}\times_S\sA_Z^* \arrow[r, hook] &X_{\sm}\times_S \sA_Z  \arrow[r, hook] &X\times_S\sA_Z 
\end{tikzcd}\]
where the left most vertical arrow is a piecewise trivial fibration with affine fibers onto the base $X_{\sm} \times \sA_Z^{*}$.

In the case of strong $n$th order deformation over linear jets, the above diagram simplifies to 
\[\begin{tikzcd}
\sL_n(X_n)\setminus \theta^{-1}(O) \arrow[d,] \arrow[r,hook] & \sL_n(X_n) \arrow[d,"{\pi|_{V}\times\theta}" left]  \arrow[r] \arrow[rd, "{\pi\times\theta}" description ]&\sL_n(X)\times_S \bA_k^n \arrow[d] \\
X_{\sm}\times_k\bG_{m}\times_k\bA_k^{n-1}\arrow[r, hook] &X_{\sm}\times_S \bA_k^n \arrow[r, hook] &X\times_S \bA_k^n
\end{tikzcd}\]
where $\bG_{m}$ is the general multiplicative group over $k$.

\begin{proposition}
Let $C$ be a curve over a field $k$. Consider a weak $n$th order deformation $C_n$ of $C$ over $\spec{k[t]/(t^{n+1})}$ such that $\sL_n(C_n)$ is irreducible and of pure dimension. Then, $n+2\leq \dim \ \pi_n^{-1}(C_{\sing}) \leq 2n$.
\end{proposition}

\begin{proof}
The proof will be exactly the same as before, and therefore we will just sketch the proof here. If $C$ is smooth, then there is nothing to prove. Therefore, assume $C$ is singular and let $V = \pi_n^{-1}(C_{\sing})$ where $\pi_n : \sL_n(C_n) \to C$ is the natural morphism. Let $\dim \ V =m  \leq 2n$ since $ \sL_n(C_n)$ is of pure dimension and irreducible.  Assume for now that $\theta_n|_V$ is not constant, then, by a linear change of coordinates if necessary, we can restrict the target of $\theta_n|_V$ to a copy of the affine line for which the image is dense. Thus, if we let $V_0 \subset \sL_n(C)$ be the fiber over the origin, $\dim \ V_0 \leq m - 1$. For any singular point $p$ on $C$, $\dim \ T_pC \geq 2$, and therefore, by Lemma 4.1 on page 18 of \cite{mus2001}, $\dim \ V_0\geq n+1$. Actually, equality is obtained in our case, so we simply solve for $m$ to obtain the lower bound. \end{proof}

\begin{example}
Consider the $2$nd order versal deformation $C_2$ of the node given by $xy - t$, with $t^3 =0$ over a field $k$. Following the computation in Example \ref{dnode3}, we obtain in exactly the same way the list of equations  
 \begin{equation*}
 \begin{split}
 a_{11}a_{21}  &= 0 \\
 a_{11}a_{22} + a_{12}a_{21} - f &= 0\\
 a_{11}a_{23}+a_{21}a_{13}+a_{12}a_{22} - g   &= 0 
 \end{split}
 \end{equation*}

These equations define $\sL_2(C_2)$ as a subvariety of $\bA_k^8$. Clearly, this is reducible with two irreducible components $W_i$ given by $a_{i1} =0$ for $i = 1, \ 2$. Each irreducible component maps in a natural way surjectively onto $\spec{k[f]}$, and is thus flat over $\spec{k[f]}$. The fiber at $f=0$ of  each component $W_{i}$ is isomorphic to $T(N_1)\times_k\bA_k^1$ where the space $T(N_1)$ studied in Example \ref{defnode}. The affine factor is coming from a free variable: $a_{23}$ is free in the case $i=1$ and $a_{13}$ is free in the case that $i=2$.  Furthermore,  we noted in Example \ref{defnode} that $T(N_1)$ is of dimension $3$. Thus, $\sL_2(C_2)$ is of pure dimension since each irreducible component must be of dimension $5$. One can now also quickly check that $\pi_2^{-1}(O) \cong \spec{k[x,y,t]/(xy-t)} \times_k \bA_k^2$, and thus $\dim \ \pi_2^{-1}(O) = 4$. 
\end{example}

\begin{example}
Consider the $2$nd order deformation $X_2$ of the node $N$ given by $xy = t^2(x+y)$, with $t^3 =0$ over a field $k$. The equations defining $\sL_2(X_2)$ are  
 \begin{equation*}
 \begin{split}
 a_{11}a_{21}  &= 0 \\
 a_{11}a_{22} + a_{12}a_{21}  &= 0\\
 a_{11}a_{23}+a_{21}a_{13}+a_{12}a_{22} - f^2(a_{11}+a_{21})   &= 0 
 \end{split}
 \end{equation*}
The fiber over the singular locus $\pi_2^{-1}(O)$ is the fiber product of the inverse image over the singular locus in $\sL_2(N)$, which is isomorphic to $N\times_k\bA_k^2$, and another copy of the affine plane given by $\spec{k[f,g]}$. This example is as far away from being strong as possible  for a locally complete intersection in that the fiber over the singular locus is an irreducible component with $ \dim \ \pi_2^{-1}(O) =  \dim \ \sL_2(N_2) =5$.
\end{example}

\begin{example}
Consider the cusp $C = \spec{k[x,y]/(y^2 -x^3)}$ and the $3$rd order versal deformation $$C_3 = \spec{k[x,y,t]/(y^2 - x^3 - t , t^4)}$$
In exactly the same manner as the previous calculations, we obtain
\begin{equation*}
\begin{split}
a_{21}^2 - a_{11}^3 &= 0\\
2a_{21}a_{22} - 3a_{11}^3a_{12} - f & = 0 \\
a_{22}^2 + a_{21}a_{23} - 3a_{11}a_{12}^2 - 2a_{11}^2a_{13} - g &= 0 \\
a_{21}a_{24} + 2a_{22}a_{23} - 2a_{11}^2a_{14} - 5 a_{11}a_{12}a_{13} - a_{12}^3 - h &= 0 
\end{split}
\end{equation*}
as equations for the auto-arc space $\sL_3(C_3)$ in $\bA_{k}^{11}$. The fiber $\pi_3^{-1}(O)$ is given by the ideal $I = (a_{22}^2 - g, 2a_{22}a_{23} - a_{12}^3 - h, a_{11}, a_{21}, f)$ from which one sees that $a_{13}, a_{14},$ and $a_{24}$ are free. Thus, there is an isomorphism
$$\pi_3^{-1}(O) \cong \spec{k[x,y,z,t,s]/ (y^2 - t, 2yz - x^3 - s)}\times_k \bA_k^3\cong \bA_k^6$$
Thus, $\pi_3^{-1}(O)$ obtains the minimum possible dimension  for a weak deformation. 
\end{example}

Considering these examples, we define
\begin{equation}
\Phi_n := \frac{\dim \ \pi^{-1}_{Z_n}(X_{\sing})}{\ell(Z_n)} - d
\end{equation}
 for any weak deformation $Y_n$ of a reduced scheme $X$ over a fat point $Z_n$. We expect that  the limit of  $\Phi_n$ to exist for any sequence of weak deformations $Y_n$ over fat points $Z_n$ given by a closed germ $(Z,P)$, and we expect this limit to fit into the inequality 
 \begin{equation}
 e(Z,P) \leq \lim_{n\to \infty} \Phi_n \leq \delta^{*}_{(Z,P)}(X)
 \end{equation}
 where $e(Z,P)$ and $\delta^{*}_{(Z,P)}(X)$ are the asymptotic defects defined in Section \ref{sec2}.

\section{Motivic volumes of auto-arc spaces}

From now on we restrict our attention to the case where $S = \spec{k}$ with $k$ a fixed algebraically closed field. We let $\var$ denote the category of varieties over $k$. We fix a $k$-scheme $Z$ and a closed point  $p$ on $Z$. Therefore, we have the $n$th order jets $Z_n:=J_p^nZ := \spec{\sO_{Z,p}/\mathfrak{m}_p^{n+1}}$ as a fixed sequenced of infinitesimal neighborhoods of $p$ on $Z$.

Let us also fix an arbitrary  sequence of infinitesimal deformations $Y_n$ of $X$ over $Z_n$ such that $Y_{n-1} \cong Y_n \times_{Z_n} Z_{n-1}$ for all $n \geq 1$, we may consider the sequence of auto-arc spaces $\sA_n := \sA_{Z_n}(Y_n)$ together with the natural induced map 
$\pi^{n}_{m} : \sA_n \to \sA_{m}$ for $n\geq m \geq 0$.  Define $\sA := \varprojlim \sA_n$ and let $\alpha_n $ denote the canonical morphism from $\sA \to \sA_n$

Let $\grot{\var}$ denote the Grothendieck ring of varieties, $\sG := \grot{\var}[\bL^{-1}]$ the localized Grothendieck ring by the Leftschetz motive, and let $\widehat{\sG}$ be the completion of $\sG$ along the dimensional filtration. We define 
\begin{equation}
\mu(X,\sY) := \lim_{n\to \infty} [\sA_n]\bL^{-d\ell-\delta_n}
\end{equation}
provided the limit in $\widehat{\sG}$ exists. If we assume $X$ is locally a complete intersection over $k$, we may decompose each term in the limit as $[\pi_n^{-1}(X_{\sing})]\bL^{-d\ell-\delta_n} + M$  where $M$ is a fixed element (i.e., not dependent on $n$). We noted then that $[\pi_n^{-1}(X_{\sing})]\bL^{-d\ell-\delta_n} \rightarrow 0$, and so we may consider the motivic measure as $M$. 

\newpage

\bibliographystyle{amsrefs}

\nocite{*}

\bibliography{FlatMap.bib}

\end{document}